\documentclass[11pt]{amsart}
\usepackage{mathrsfs,latexsym,amsfonts,amssymb}
\newtheorem{theorem}{Theorem}[section]
\newtheorem{lemma}[theorem]{Lemma}
\newtheorem{corollary}[theorem]{Corollary}

\theoremstyle{definition}
\newtheorem{definition}[theorem]{Definition}
\newtheorem{proposition}[theorem]{Proposition}
\theoremstyle{remark}
\newtheorem{remark}[theorem]{Remark}

\begin{document}

\title[Quotient with respect to admissible $L$-subgyrogroups]
{Quotient with respect to admissible $L$-subgyrogroups}

\author{Meng Bao}
\address{(Meng Bao): 1. School of mathematics and statistics,
Minnan Normal University, Zhangzhou 363000, P. R. China}
\email{mengbao95213@163.com}

\author{Fucai Lin*}
\address{(Fucai Lin): School of mathematics and statistics,
Minnan Normal University, Zhangzhou 363000, P. R. China}
\email{linfucai2008@aliyun.com; linfucai@mnnu.edu.cn}

\thanks{The authors are supported by the NSFC (Nos. 11571158), the Natural Science Foundation of Fujian Province (No. 2017J01405) of China, the Program for New Century Excellent Talents in Fujian Province University, the Institute of Meteorological Big Data-Digital Fujian and Fujian Key Laboratory of Data Science and Statistics.\\
*corresponding author}


\begin{abstract}
The concept of gyrogroups, with a weaker algebraic structure without associative law, was introduced under the background of $c$-ball of relativistically admissible velocities with Einstein velocity addition. A topological gyrogroup is just a gyrogroup endowed with a compatible topology such that the multiplication is jointly continuous and the inverse is continuous. This concept is a good generalization of a topological group. In this paper, we are going to establish that for a locally compact admissible $L$-subgyrogroup $H$ of a strongly
topological gyrogroup $G$, the natural quotient mapping $\pi$ from $G$ onto the quotient space $G/H$
has some nice local properties, such as, local compactness, local pseudocompactness, local paracompactness, etc. Finally, we prove that each locally paracompact strongly topological gyrogroup is paracompact.
\end{abstract}

\maketitle
\section{Introduction}
The concept of a gyrogroup was discovered by Ungar in researching $c$-ball of relativistically admissible velocities with the Einstein velocity addition in \cite{UA}. The Einstein velocity addition $\oplus_{E}$ in the $c$-ball is given by the equation
$$\mathbf{u}\oplus_{E}\mathbf{v}=\frac{1}{1+\frac{\langle \mathbf{u}, \mathbf{v}\rangle}{c^{2}}}\left\{\mathbf{u}+\frac{1}{\gamma_{\mathbf{u}}}\mathbf{v}+\frac{1}{c^{2}}\frac{\gamma_{\mathbf{u}}}{1+\gamma_{\mathbf{u}}}\langle\mathbf{u}, \mathbf{v}\rangle\mathbf{u}\right\},$$where $\mathbf{u}, \mathbf{v}\in\mathbb{R}_{c}^{3}=\{\mathbf{v}\in\mathbb{R}^{3}: \parallel\mathbf{v}\parallel<c\}$ and $\gamma_{\mathbf{u}}$ is the Lorentz factor given by $$\gamma_{\mathbf{u}}=\frac{1}{\sqrt{1-\frac{\parallel\mathbf{u}\parallel^{2}}{c^{2}}}}.$$ The system $(\mathbb{R}_{c}^{3},
\oplus_{E})$ does not form a group because $\oplus_{E}$ is neither associative nor commutative.
However, it has a weaker algebraic structure than a group, and now it has been a very interesting topic during the past few years, see \cite{AW2020,FM,FM1,LF,LF1,LF2,LF3,SL,ST,ST1,ST2,UA2002,WAS2020}. In particular, Atiponrat \cite{AW} defined the concept of topological gyrogroups, and then Cai, Lin and He in \cite{CZ} proved that every topological gyrogroup is a rectifiable space. In 2019, Bao and Lin \cite{BL} defined the concept of strongly topological gyrogroups and proved that every feathered strongly topological gyrogroup is paracompact, which implies that every feathered strongly topological gyrogroup is a $D$-space. Moreover, Bao and Lin also proved that every locally compact $NSS$-gyrogroup is first-countable and each Lindel\"{o}f $P$-gyrogroup is Ra\v{\i}kov complete.

In this paper, we mainly discuss perfect mappings between strongly topological gyrogroups. In particular, we prove that if $(G, \tau, \oplus)$ is a strongly topological gyrogroup with a symmetric neighborhood base $\mathscr U$ at $0$, $H$ is a locally compact admissible subgyrogroup generated from $\mathscr U$, and $\pi :G\rightarrow G/H$ is the natural quotient mapping from $G$ onto the quotient space $G/H$, then there exists an open neighborhood $U$ of the identity element $0$ such that $\pi (\overline{U})$ is closed in $G/H$ and the restriction of $\pi $ to $\overline{U}$ is a perfect mapping from $\overline{U}$ onto the subspace $\pi (\overline{U})$. Therefore, if $H$ is a locally compact admissible $L$-subgyrogroup of a strongly topological gyrogroup $G$ such that the quotient space $G/H$ has some nice properties which are inherited by regular closed sets and preserved by perfect preimages, such as, locally compact, locally countablely compact, locally pseudocompact, locally paracompact and so on, then the space $G$ has the same properties too. These results extend some well known results for topological groups. Further, we prove that each locally paracompact strongly topological gyrogroup is paracompact.

\smallskip
\section{Preliminaries}
In this section, we introduce necessary notation, terminology and facts about topological gyrogroups.

Throughout this paper, all topological spaces are assumed to be
Hausdorff, unless otherwise is explicitly stated. Let $\mathbb{N}$ be the set of all positive integers and $\omega$ the first infinite ordinal. Let $X$ be a topological space $X$, and let $A$ be a subset of $X$.
  The {\it closure} of $A$ in $X$ is denoted by $\overline{A}$ and the
  {\it interior} of $A$ in $X$ is denoted by $\mbox{Int}(A)$. The readers may consult \cite{AA, E, linbook} for notation and terminology not explicitly given here.

\begin{definition}\cite{AW}
Let $G$ be a nonempty set, and let $\oplus: G\times G\rightarrow G$ be a binary operation on $G$. Then the pair $(G, \oplus)$ is called a {\it groupoid}. A function $f$ from a groupoid $(G_{1}, \oplus_{1})$ to a groupoid $(G_{2}, \oplus_{2})$ is called a {\it groupoid homomorphism} if $f(x\oplus_{1}y)=f(x)\oplus_{2} f(y)$ for any elements $x, y\in G_{1}$. Furthermore, a bijective groupoid homomorphism from a groupoid $(G, \oplus)$ to itself will be called a {\it groupoid automorphism}. We write $\mbox{Aut}(G, \oplus)$ for the set of all automorphisms of a groupoid $(G, \oplus)$.
\end{definition}

\begin{definition}\cite{UA}
Let $(G, \oplus)$ be a groupoid. The system $(G,\oplus)$ is called a {\it gyrogroup}, if its binary operation satisfies the following conditions:

\smallskip
(G1) There exists a unique identity element $0\in G$ such that $0\oplus a=a=a\oplus0$ for all $a\in G$.

\smallskip
(G2) For each $x\in G$, there exists a unique inverse element $\ominus x\in G$ such that $\ominus x \oplus x=0=x\oplus (\ominus x)$.

\smallskip
(G3) For all $x, y\in G$, there exists $\mbox{gyr}[x, y]\in \mbox{Aut}(G, \oplus)$ with the property that $x\oplus (y\oplus z)=(x\oplus y)\oplus \mbox{gyr}[x, y](z)$ for all $z\in G$.

\smallskip
(G4) For any $x, y\in G$, $\mbox{gyr}[x\oplus y, y]=\mbox{gyr}[x, y]$.
\end{definition}

Note that a group is a gyrogroup $(G,\oplus)$ such that $\mbox{gyr}[x,y]$ is the identity function for all $x, y\in G$. The definition of a subgyrogroup is given as follows.

\begin{definition}\cite{ST}
Let $(G,\oplus)$ be a gyrogroup. A nonempty subset $H$ of $G$ is called a {\it subgyrogroup}, denoted
by $H\leq G$, if the following statements hold:

\smallskip
(i) The restriction $\oplus| _{H\times H}$ is a binary operation on $H$, i.e. $(H, \oplus| _{H\times H})$ is a groupoid.

\smallskip
(ii) For any $x, y\in H$, the restriction of $\mbox{gyr}[x, y]$ to $H$, $\mbox{gyr}[x, y]|_{H}$ : $H\rightarrow \mbox{gyr}[x, y](H)$, is a bijective homomorphism.

\smallskip
(iii) $(H, \oplus|_{H\times H})$ is a gyrogroup.

\smallskip
Furthermore, a subgyrogroup $H$ of $G$ is said to be an {\it $L$-subgyrogroup} \cite{ST}, denoted
by $H\leq_{L} G$, if $\mbox{gyr}[a, h](H)=H$ for all $a\in G$ and $h\in H$.

\end{definition}

\begin{definition}\cite{AW}
A triple $(G, \tau, \oplus)$ is called a {\it topological gyrogroup} if the following statements hold:

\smallskip
(1) $(G, \tau)$ is a topological space.

\smallskip
(2) $(G, \oplus)$ is a gyrogroup.

\smallskip
(3) The binary operation $\oplus: G\times G\rightarrow G$ is jointly continuous while $G\times G$ is endowed with the product topology, and the operation of taking the inverse $\ominus (\cdot): G\rightarrow G$, i.e. $x\rightarrow \ominus x$, is also continuous.
\end{definition}

\begin{definition}\cite{BL}
Let $G$ be a topological gyrogroup. We say that $G$ is a {\it strongly topological gyrogroup} if there exists a neighborhood base $\mathscr U$ of $0$ such that, for every $U\in \mathscr U$, $\mbox{gyr}[x, y](U)=U$ for any $x, y\in G$. For convenience, we say that $G$ is a strongly topological gyrogroup with neighborhood base $\mathscr U$ of $0$. Clearly, we may assume that $U$ is symmetric for each $U\in\mathscr U$.
\end{definition}

\begin{remark}
By \cite[Example 3.1]{BL}, there exists a strongly topological gyrogroup which is not a topological group. Moreover, the authors gave an example in \cite{BL} to show that there exists a strongly topological gyrogroup which has an infinite $L$-subgyrogroup, see \cite[Example 3.2]{BL}.
\end{remark}

\begin{definition}\cite{UA}
Let $(G,\oplus)$ be a gyrogroup, and let $x\in G$. We define the {\it left gyrotranslation} by $x$ to be the function $L_{x} : G\rightarrow G$ such that $L_{x} (y)=x\oplus y$ for any $y\in G$. In addition, the {\it right gyrotranslation} by $x$ is defined to be the function $R_{x} (y)=y\oplus x$ for any $y\in G$. Clearly, all left and right gyrotranslations are homeomorphic mappings from $G$ onto itself.
\end{definition}

\begin{definition}\cite{E}
A continuous mapping $f: X\rightarrow Y$ is called {\it closed} (resp. {\it open}) if for every closed (resp. open) set $A\subset X$ the image $f(A)$ is closed (resp. {\it open}) in $Y$.
\end{definition}

\begin{definition}\cite{E}
A continuous mapping $f:X\rightarrow Y$ is {\it perfect} if $f$ is a closed mapping and all fibers $f^{-1}(y)$ are compact subsets of $X$.
\end{definition}

\begin{definition}\cite{AA}
Let $G$ be a topological gyrogroup, and $A$ and $B$ be subsets of $G$. Let us say that $A$ and $B$ are {\it cross-complementary} in $G$ if $G=A\oplus B=\{a\oplus b:a\in A,b\in B\}$. Suppose that $\mathscr P$ is a topological property, and $A$ is a subset of $G$. We say that $A$ has a {\it $\mathscr P$-grasp} on $G$ if there exists a subset $B$ of $G$ such that $B$ is cross-complementary to $A$ and has the property $\mathscr P$.
\end{definition}

Finally, we give some facts about gyrogroups and topological gyrogroups, which are important in our proofs.

\begin{proposition}\cite{ST}.
Let $(G, \oplus)$ be a gyrogroup, and let $H$ be a nonempty subset of $G$. Then $H$ is a subgyrogroup if and only if the following statements are true:

\smallskip
(1) For any $x\in H$, $\ominus x\in H$, and

\smallskip
(2) for any $x, y\in H$, $x\oplus y\in H$.
\end{proposition}

\begin{proposition}\cite{UA}\label{a}.
Let $(G, \oplus)$ be a gyrogroup. Then for any $x, y, z\in G$,

\smallskip
(1) $(\ominus x)\oplus (x\oplus y)=y$.

\smallskip
(2) $(x\oplus (\ominus y))\oplus \mbox{gyr}[x, \ominus y](y)=x$.

\smallskip
(3) $(x\oplus \mbox{gyr}[x, y](\ominus y))\oplus y=x$.

\smallskip
(4) $\mbox{gyr}[x, y](z)=\ominus (x\oplus y)\oplus (x\oplus (y\oplus z))$.
\end{proposition}

 \maketitle
\section{Perfect mappings between strongly topological gyrogroups }
In this section, we mainly demonstrate that for a locally compact admissible $L$-subgyrogroup of a strongly topological gyrogroup $G$, the natural mapping $\pi$ from $G$ onto the quotient space $G/H$ has some nice local properties. First, We recall the following concept of the coset space of a topological gyrogroup.

Let $(G, \tau, \oplus)$ be a topological gyrogroup and $H$ an $L$-subgyrogroup of $G$. It follows from \cite[Theorem 20]{ST} that $G/H=\{a\oplus H:a\in G\}$ is a partition of $G$. We denote by $\pi$ the mapping $a\mapsto a\oplus H$ from $G$ onto $G/H$. Clearly, for each $a\in G$, we have $\pi^{-1}\{\pi(a)\}=a\oplus H$.
Denote by $\tau (G)$ the topology of $G$. In the set $G/H$, we define a family $\tau (G/H)$ of subsets as follows: $$\tau (G/H)=\{O\subset G/H: \pi^{-1}(O)\in \tau (G)\}.$$

In \cite{BL}, the authors proved the following lemma.

\begin{lemma}\label{t00000}\cite{BL}
Let $(G, \tau, \oplus)$ be a topological gyrogroup and $H$ an $L$-subgyrogroup of $G$. Then the natural homomorphism $\varphi$ from a topological gyrogroup $G$ to its quotient topology on $G/H$ is an open and continuous mapping.
\end{lemma}

The following concept of an admissible subgyrogroup of a strongly topological gyrogroup was first introduced in \cite{BL1}, which plays an important role in this paper.

A subgyrogroup $H$ of a topological gyrogroup $G$ is called {\it admissible} if there exists a sequence $\{U_{n}:n\in \omega\}$ of open symmetric neighborhoods of the identity $0$ in $G$ such that $U_{n+1}\oplus (U_{n+1}\oplus U_{n+1})\subset U_{n}$ for each $n\in \omega$ and $H=\bigcap _{n\in \omega}U_{n}$. If $G$ is a strongly topological gyrogroup with a symmetric neighborhood base $\mathscr U$ at $0$ and each $U_{n}\in \mathscr U$, we say that the admissible topological subgyrogroup is generated from $\mathscr U$.

The following thereom shows that the admissible topological subgyrogroup generated from $\mathscr U$ of a strongly topological gyrogroup is a closed $L$-subgyrogroup.

\begin{theorem}\label{3dl5}
Suppose that $(G, \tau, \oplus)$ is a strongly topological gyrogroup with a symmetric neighborhood base $\mathscr U$ at $0$. Then each admissible topological subgyrogroup $H$ generated from $\mathscr U$ is a closed $L$-subgyrogroup of $G$.
\end{theorem}

\begin{proof}
Since $H$ is generated from $\mathscr U$, there exits a sequence $\{U_{n}:n\in \omega\}$ of open symmetric neighborhoods of the identity $0$ in $G$ such that $U_{n}\in \mathscr U$, $U_{n+1}\oplus (U_{n+1}\oplus U_{n+1})\subset U_{n}$ for each $n\in \omega$ and $H=\bigcap _{n\in \omega}U_{n}$. Obviously, for each $n\in \omega$, $$\overline{U_{n+1}}\subset U_{n+1}\oplus U_{n+1}\subset U_{n+1}\oplus (U_{n+1}\oplus U_{n+1})\subset U_{n},$$ so $H$ is a closed subgyrogroup of $G$. Moreover, for every $x, y\in G$, we have $$gyr[x, y](H)=gyr[x, y](\bigcap _{n\in \omega}U_{n})\subset \bigcap _{n\in \omega}gyr[x, y](U_{n})\subset \bigcap _{n\in \omega}U_{n}=H,$$ and it follows from \cite[Proposition 2.6]{ST} that $gyr[x, y](H)=H$. Therefore, it follows from the definition of $L$-subgyrogroups that $H$ is an $L$-subgyrogroup. Hence $H$ is a closed $L$-subgyrogroup of $G$.
\end{proof}

\begin{theorem}\label{dl1}
Suppose that $(G, \tau, \oplus)$ is a strongly topological gyrogroup with a symmetric neighborhood base $\mathscr U$ at $0$, $H$ is an admissible subgyrogroup generated from $\mathscr U$ and $P$ is a closed symmetric subset of $G$ such that $P$ contains an open neighborhood of $0$ in $G$, and $\overline{P\oplus (P\oplus P)}\cap H$ is compact. Then the restriction $f$ of $\pi$ to $P$ is a perfect mapping from $P$ onto the subspace $\pi(P)$ of $G/H$, where $\pi: G\rightarrow G/H$ is the natural quotient mapping from $G$ onto the quotient space $G/H$.
\end{theorem}

\begin{proof}
It is clear that $f$ is continuous. First, we claim that $f^{-1}f(a)$ is compact for each $a\in P$. Indeed, from the definition of $f$, we have $f^{-1}f(a)=(a\oplus H)\cap P$. Since the left gyrotranslation ia a homeomorphism, the subspace $(a\oplus H)\cap P$ and $H\cap ((\ominus a)\oplus P)$ are homeomorphic, thus both of them are closed in $G$ by Theorem~\ref{3dl5}. Since $\ominus a\in \ominus P=P$, we have $$H\cap ((\ominus a)\oplus P)\subset H\cap (P\oplus P)\subset \overline{P\oplus (P\oplus P)}\cap H.$$ Hence, $H\cap ((\ominus a)\oplus P)$ is compact and so is the set $f^{-1}f(a)$.

It remains to verify that the $f$ is a closed mapping. Let us fix any closed subset $M$ of $P$ and let $a$ be an any point of $P$ such that $f(a)\in \overline{f(M)}$. We prove that $f(a)\in f(M)$. Suppose not. Then $(a\oplus H)\cap (M\oplus H)\cap P=\emptyset$. Since $H$ is an admissible subgyrogroup generated from $\mathscr U$, by Theorem~\ref{3dl5}, $H$ is an $L$-subgyrogroup. Hence it follows from \cite[Theorem 2.27]{UA2005} that
\begin{eqnarray}
(a\oplus H)\oplus H&=&a\oplus(\bigcup_{x, y\in H}x\oplus \mbox{gyr}[x, a]y)\nonumber\\
&\subset&a\oplus(\bigcup_{x, y\in H}x\oplus \mbox{gyr}^{-1}[a, x]y)\nonumber\\
&\subset&a\oplus(H\oplus H)\nonumber\nonumber\\
&=&a\oplus H,\nonumber
\end{eqnarray}
and thus $(a\oplus H)\oplus H=a\oplus H$. Then $(a\oplus H)\cap M\cap P=\emptyset$, and $(a\oplus H)\cap\overline{(P\oplus P)}\cap M=\emptyset$ since $M\subset P$.
Obviously, $(a\oplus H)\cap \overline{(P\oplus P)}$ is compact. Since $M$ is a closed and disjoint from the compact subset $(a\oplus H)\cap \overline{(P\oplus P)}$, it easily verify that there exists an open neighborhood $W\in \mathscr U$ such that $W\subset P$ and $(W\oplus ((a\oplus H)\cap \overline{(P\oplus P)}))\cap M=\emptyset $.

Since the quotient mapping $\pi $ is open by Lemma~\ref{t00000} and $W\oplus a$ is an open neighborhood of $a$, the set $\pi (W\oplus a)$ is an open neighborhood of $\pi (a)$ in $G/H$. Therefore, the set $\pi (W\oplus a)\cap \pi (M)\neq \emptyset $ and we can fix $m\in M$ and $y\in W$ such that $\pi (m)=\pi (y\oplus a)$, that is, $m\in (y\oplus a)\oplus H$. Then,
\begin{eqnarray}
(y\oplus a)\oplus H&=&y\oplus (a\oplus gyr[a,y](H))\nonumber\\
&=&y\oplus (a\oplus gyr[a,y](\bigcap _{n\in \omega}V_{n}))\nonumber\\
&\subset &y\oplus (a\oplus \bigcap _{n\in \omega}gyr[a,y](V_{n}))\nonumber\\
&=&y\oplus (a\oplus \bigcap _{n\in \omega}V_{n})\nonumber\\
&=&y\oplus (a\oplus H).\nonumber
\end{eqnarray}
Hence, there exists an $h\in H$ such that $a\oplus h=\ominus y\oplus m$. Since $\ominus y\in \ominus W=W\subset P$ and $m\in M\subset P$, we have that $a\oplus h=\ominus y\oplus m\in \overline{(P\oplus P)}$. In addition, $a\oplus h\in a\oplus H$. Hence, $a\oplus h\in ((a\oplus H)\cap \overline{(P\oplus P)})$ and $m\in W\oplus ((a\oplus H)\cap \overline{(P\oplus P)})$. Thus, $M\cap (W\oplus ((a\oplus H)\cap \overline{(P\oplus P)}))\neq \emptyset$, which is a contradiction.

Therefore, $f(a)\in f(M)$ and $f(M)$ is closed in $f(P)$. Then, since $a\in P$ is arbitrarily taken, the mapping $f$ is perfect.
\end{proof}

\begin{theorem}\label{dl2}
Suppose that $(G, \tau, \oplus)$ is a strongly topological gyrogroup with a symmetric neighborhood base $\mathscr U$ at $0$, and suppose that $H$ is a locally compact admissible subgyrogroup generated from $\mathscr U$. Then there exists an open neighborhood $U$ of the identity element $0$ such that $\pi (\overline{U})$ is closed in $G/H$ and the restriction of $\pi $ to $\overline{U}$ is a perfect mapping from $\overline{U}$ onto the subspace $\pi (\overline{U})$, where $\pi: G\rightarrow G/H$ is the natural quotient mapping from $G$ onto the quotient space $G/H$.
\end{theorem}

\begin{proof}
First, $H$ is closed in $G$ by \cite[Proposition 2.7]{LF3}. Since $H$ is locally compact, we can find an open neighborhood $V$ of $0$ in $G$ such that $\overline{V\cap H}$ is compact. By the regularity of $G$, we can choose an open neighborhood $W$ of $0$ such that $\overline{W}\subset V$. Hence $\overline{W}\cap H$ is compact. Let $U_{0}$ be an arbitrary symmetric open neighborhood of $0$ such that $U_{0}\oplus (U_{0}\oplus U_{0})\subset W$. By the joint continuity, we have $\overline{U_{0}}\oplus (\overline{U_{0}}\oplus \overline{U_{0}})\subset \overline{U_{0}\oplus (U_{0}\oplus U_{0})}$. Then the set $P=\overline{U_{0}}$ satisfies all restrictions on $P$ in Theorem \ref{dl1}. It follows from by Theorem \ref{dl1} that the restriction of $\pi $ to $P$ is a perfect mapping from $P$ onto the subspace $\pi (P)$.

It follows from Lemma \ref{t00000} that $\pi $ is an open mapping, the set $\pi (U_{0})$ is open in $G/H$. Then we claim that $G/H$ is regular. Indeed, it is obvious that $G/H$ is $T_{1}$ since $H$ is closed. By \cite[Proposition 4.4]{BL1}, $G/H$ is a homogeneous space, hence it suffices to prove that for each open neighborhood $U$ of 0 in $G$ there exists an open neighborhood $V$ of 0 in $G$ such that $\overline{\pi(V)}\subset \pi(U)$. Then for any open neighborhood $U$ of 0 in $G$, one can pick a symmetric open neighborhood $V$ of 0 in $G$ such that $V\oplus V\subset U$. It is easily verified that $\overline{\pi(V)}\subset \pi(U)$.

Since the space $G/H$ is regular, there exists an open neighborhood $V_{0}$ of $\pi (0)$ in $G/H$ such that $\overline{V_{0}}\subset \pi (U_{0})$. Hence $U=\pi ^{-1}(V_{0})\cap U_{0}$ is an open neighborhood of $0$ contained in $P$ such that the restriction $f$ of $\pi$ to $\overline{U}$ is a perfect mapping from $\overline{U}$ onto the subspace $\pi (\overline{U})$. Furthermore, $\pi (\overline{U})$ is closed in $\pi (P)$, and $\pi (\overline{U})\subset \overline{V_{0}}\subset \pi (U_{0})\subset \pi (P)$. Then $\pi (\overline{U})$ is closed in $\overline{V_{0}}$, so that $\pi (\overline{U})$ is closed in $G/H$.
\end{proof}

Recall that a space $X$ is said to be {\it zero-dimensional}
if it has a base consisting of sets which are both open and closed in $X$.

\begin{theorem}
Suppose that $(G, \tau, \oplus)$ is a zero-dimensional strongly topological gyrogroup with a symmetric neighborhood base $\mathscr U$ at $0$, and $H$ is a locally compact admissible subgyrogroup generated from $\mathscr U$. Then the quotient space $G/H$ is also zero-dimensional.
\end{theorem}

\begin{proof}
Let $\pi :G\rightarrow G/H$ be the natural quotient mapping from $G$ onto the quotient space $G/H$. From Theorem \ref{dl2}, there exists an open neighborhood $U$ of the identity element $0$ of $G$ such that $\pi (\overline{U})$ is closed in $G$ and the restriction of $\pi $ to $\overline{U}$ is a perfect mapping from $\overline{U}$ onto the subspace $\pi (\overline{U})$. Let $W$ be an open neighborhood of $\pi(0)$ in $G/H$. Since the space $G$ is zero-dimensional, there exists an open and closed neighborhood $V$ of $0$ such that $V\subset U\cap \pi ^{-1}(W)$. Hence $\pi(V)$ is an open subset of $G/H$. Moreover, since the restriction of $\pi $ to $\overline{U}$ is a closed mapping and $\pi (\overline{U})$ is closed in $G/H$, $\pi(V)$ is closed in $G/H$. It is obvious that $\pi (V)\subset W$. Thus $G/H$ is zero-dimensional.
\end{proof}

Recall that a closed set in a space is called {\it regular closed} if it is the closure of an open subset of this space.

\begin{corollary}\label{tl1}
Assume that $\mathscr P$ is a topological property preserved by preimages of spaces under perfect mappings (in the class of completely regular spaces) and also inherited by regular closed sets. Assume further that $(G, \tau, \oplus)$ is a strongly topological gyrogroup with a symmetric neighborhood base $\mathscr U$ at $0$, $H$ is a locally compact admissible subgyrogroup generated from $\mathscr U$, and the quotient space $G/H$ has the property $\mathscr P$. Then there exists an open neighborhood $U$ of the identity element $0$ such that $\overline{U}$ has the property $\mathscr P$.
\end{corollary}

Given a space $X$ and a property $\mathscr P$, if each point $x$ of $X$ has an open neighborhood $U(x)$ such that $\overline{U(x)}$ has $\mathscr P$, then we say that $X$ has the property $\mathscr P$ {\it locally} \cite{AA}. It is well known that local compactness, countable compactness, pseudocompactness, paracompactness, the Lindel\"{o}f property, $\sigma $-compactness, $\check{C}$ech-completeness, the Hewitt-Nachbin completeness, and the property of being a $k$-space are all inherited by regular closed sets and preserved by perfect preimages, see \cite[Sections 3.7, 3.10, 3.11]{E}. Therefore, we have the following corollaries.

\begin{corollary}\label{tl1}
Let $(G, \tau, \oplus)$ be a strongly topological gyrogroup with a symmetric neighborhood base $\mathscr U$ at $0$, and let $H$ be a locally compact admissible subgyrogroup generated from $\mathscr U$. If the quotient space $G/H$ has one of the following properties

\smallskip
(1) $G/H$ is locally compact;

\smallskip
(2) $G/H$ is locally countably compact;

\smallskip
(3) $G/H$ is locally pseudocompact;

\smallskip
(4) $G/H$ is locally $\sigma $-compact;

\smallskip
(5) $G/H$ is locally paracompact;

\smallskip
(6) $G/H$ is locally Lindel\"{o}f;

\smallskip
(7) $G/H$ is locally $\check{C}$ech-complete;

\smallskip
(8) $G/H$ is locally realcompact,\\
\smallskip
then $G$ also has the same property.
\end{corollary}

\begin{corollary}
Let $(G, \tau, \oplus)$ be a strongly topological gyrogroup with a symmetric neighborhood base $\mathscr U$ at $0$, and let $H$ be a locally compact admissible subgyrogroup generated from $\mathscr U$. If the quotient space $G/H$ is a $k$-space, then $G$ is also a $k$-space.
\end{corollary}

\begin{proof}
Since the property of being a $k$-space is invariant under taking perfect preimages and a locally $k$-space is a $k$-space \cite[Section 3.3]{E}, the result follows from Corollary \ref{tl1} immediately.
\end{proof}

Recall that the tightness of a space $X$ is the minimal cardinal $\tau \geq \omega$ with the property that for every point $x\in X$ and every set $P\subset X$ with $x\in \overline{P}$, there exists a subset $Q$ of $P$ such that $|Q|\leq \tau$ and $x\in \overline{Q}$, see \cite{AA}.

\begin{lemma}\cite[Proposition 4.7.16]{AA}\label{4yl1}
Suppose that $f:X\rightarrow Y$ is a closed continuous mapping of a regular space $X$ onto a space $Y$ of countable tightness. Suppose further that the tightness of every fiber $f^{-1}(y)$, for $y\in Y$, is countable. Then the tightness of $X$ is also countable.
\end{lemma}

\begin{theorem}
Let $(G, \tau, \oplus)$ be a strongly topological gyrogroup with a symmetric neighborhood base $\mathscr U$ at $0$, and let $H$ be a locally compact metrizable admissible subgyrogroup generated from $\mathscr U$. If the tightness of the quotient space $G/H$ is countable, then the tightness of $G$ is also countable.
\end{theorem}

\begin{proof}
By Theorem~\ref{dl2}, there is an open neighborhood $U$ of the identity $0$ in $G$ such that $\overline{U}$ is a preimage of a space of countable tightness under a perfect mapping with metrizable fibers. Then, by Lemma \ref{4yl1}, the tightness of $\overline{U}$ is also countable. Since $U$ is a non-empty open subset of the homogeneous space $G$, the tightness of $G$ is countable.
\end{proof}

A topological gyrogroup is {\it feathered} if it contains a non-empty compact set $K$ of countable character in $G$. In \cite{BL}, the authors proved that every feathered strongly topological gyrogroup is paracompact. Hence we have the following result.

\begin{theorem}
Let $(G, \tau, \oplus)$ be a strongly topological gyrogroup with a symmetric neighborhood base $\mathscr U$ at $0$, and let $H$ be a locally compact admissible subgyrogroup generated from $\mathscr U$. If the quotient space $G/H$ is a feathered space, then $G$ is a paracompact feathered space.
\end{theorem}

\begin{proof}
By Theorem \ref{dl2}, there exists an open neighborhood $U$ of the identity element $0$ in $G$ such that $\overline{U}$ is a preimage of a closed subset of $G/H$ under a perfect mapping. Moreover, since the class of feathered spaces is closed under taking closed subspaces, it follows from \cite[Proposition 4.3.36]{AA} that $\overline{U}$ is a feathered space. Therefore, $U$ contains a non-empty compact subspace $F$ with a countable base of neighborhoods in $G$, thus $G$ is a paracompact feathered space by \cite[Corollary 3.15]{BL}.
\end{proof}

A mapping $f:X\rightarrow Y$ is {\it locally perfect} \cite{AA} if, for each $x\in X$, there is an open neighborhood $U$ of $x$ such that the restriction of $f$ to the closure of $U$ is a perfect mapping from $\overline{U}$ to $Y$. Moreover, a mapping $f:X\rightarrow Y$ is said to be {\it compact (Lindel\"{o}f)-covering} or {\it $k$-covering} \cite{AA, E} if, for each compact (Lindel\"{o}f) subspace $F$ of $Y$, there is a compact (Lindel\"{o}f) subspace $K$ of $X$ such that $f(K)=F$.

Since every open locally perfect mapping $f$ from a space $X$ onto a space $Y$ is compact-covering, we have the following result.

\begin{theorem}\label{4dl1}
Let $(G, \tau, \oplus)$ be a strongly topological gyrogroup with a symmetric neighborhood base $\mathscr U$ at $0$, and let $H$ be a locally compact admissible subgyrogroup generated from $\mathscr U$. Then the quotient mapping $\pi :G\rightarrow G/H$ is compact-covering.
\end{theorem}

\begin{proof}
By Theorem \ref{dl2}, the quotient mapping $\pi $ is locally perfect. From Lemma \ref{t00000}, it follows that the mapping $\pi $ is open. Hence $\pi$ is a compact-covering mapping.
\end{proof}

Finally, we give some applications of Theorem~\ref{4dl1}.

\begin{theorem}\label{4dl2}
Let $(G, \tau, \oplus)$ be a strongly topological gyrogroup with a symmetric neighborhood base $\mathscr U$ at $0$, and let $H$ be a locally compact admissible subgyrogroup generated from $\mathscr U$. Then $H$ has a compact grasp on $G$ if and only if the quotient space $G/H$ is compact.
\end{theorem}

\begin{proof}
Suppose that $G=F\oplus H$, where $F$ is compact. Hence $G/H=\pi (G)=\pi (F)$, where $\pi$ is the quotient mapping from $G$ onto $G/H$. Then it follows from Lemma \ref{t00000} that $\pi$ is continuous and the space $G/H$ is compact. The rest of proof just follows from Theorem \ref{4dl1}.
\end{proof}

Since each open locally perfect mapping $f$ from a space $X$ onto a space $Y$ is also Lindel\"{o}f-covering, we use the same proof like in Theorem \ref{4dl1} and Theorem \ref{4dl2} to obtain the following two corollaries.

\begin{corollary}
Let $(G, \tau, \oplus)$ be a strongly topological gyrogroup with a symmetric neighborhood base $\mathscr U$ at $0$, and let $H$ be a locally compact admissible subgyrogroup generated from $\mathscr U$. Then the quotient mapping $\pi :G\rightarrow G/H$ is Lindel\"{o}f-covering.
\end{corollary}

\begin{corollary}
Let $(G, \tau, \oplus)$ be a strongly topological gyrogroup with a symmetric neighborhood base $\mathscr U$ at $0$, and let $H$ be a locally compact admissible subgyrogroup generated from $\mathscr U$. Then $H$ has a Lindel\"{o}f grasp on $G$ if and only if the quotient space $G/H$ is Lindel\"{o}f.
\end{corollary}

 \maketitle
\section{Local paracompactness of strongly topological gyrogroups}
In the section, we will prove that every locally paracompact strongly topological gyrogroup is paracompact, which generalizes the important result in \cite{AU}.

Let $X$ be a nonempty set. The diagonal of $X$ is the set $\Delta =\{(x,x):x\in X\}$. Every set $V\subset X\times X$ that contains $\Delta$ and satisfies the condition $V=-V$ is called {\it an entourage of the diagonal} \cite{E}. The family of all entourages of the diagonal $\Delta \subset X\times X$ will be denoted by $\mathscr D_{X}$. Let $x_{0}\in X$ and $V\in \mathscr D_{X}$, and the set $$B(x_{0}, V)=\{x\in X:|x_{0}-x|<V\}$$ is called {\it the ball with the centre $x_{0}$ and radius $V$} or, briefly, {\it the $V$-ball about $x_{0}$}, see \cite{E}. Let $\mathscr C(V)=\{B(x, V)\}_{x\in X}$ for each $V\in \mathscr D_{X}$. Clearly, each $\mathscr C(V)$ is a cover of $X$. Let $\mathscr V$ be a uniformity on $X$. Then any cover of the set $X$ which has a refinement of the form $\mathscr C(V)$, where $V\in \mathscr V$, is called {\it uniform with respect to $\mathscr V$} \cite{E}. From \cite[Section 8.1]{E}, the collection $\mathcal{C}$ of all covers of a set $X$ which are uniform with respect to a uniformity $\mathscr V$ on the set $X$ has the following properties:

\smallskip
(UC1) If $\mathscr A\in \mathcal{C}$ and $\mathscr A$ is a refinement of a cover $\mathscr B$ of the set $X$, then $\mathscr B\in \mathcal{C}$;

\smallskip
(UC2) For any $\mathscr A_{1},\mathscr A_{2}\in \mathcal{C}$, there exists an $\mathscr A\in \mathcal{C}$ which is a refinement of both $\mathscr A_{1}$ and $\mathscr A_{2}$;

\smallskip
(UC3) For every $\mathscr A\in \mathcal{C}$, there exists a $\mathscr B\in \mathcal{C}$ which is a star refinement of $\mathscr A$;

\smallskip
(UC4) For every pair $x, y$ of distinct points of $X$, there exists an $\mathscr A\in \mathcal{C}$ such that no member of $\mathscr A$ contains both $x$ and $y$.

Suppose that $(G, \tau ,\oplus)$ is a strongly topological gyrogroup with a symmetric open neighborhood base $\mathscr U$ at $0$, then every member $V\in \mathscr U$ determines a cover of $G$: $$\mathcal{C}_{l}(V)=\{x\oplus V\}_{x\in G}.$$ Denote by $\mathcal{C}_{l}$ the collection of all covers of $G$ which have a refinement of the form $\mathcal{C}_{l}(V)$, where $V\in \mathscr U$. First, we give the following theorem.

\begin{theorem}\label{th9}
Suppose that $(G, \tau, \oplus)$ is a strongly topological gyrogroup with a symmetric open neighborhood base $\mathscr U$ at $0$, then $\mathcal{C}_{l}$ has the properties (UC1)-(UC4), and generates a uniformity on the set $G$. Moreover, the topology induced by the uniformity coincides with the original topology of $G$.
\end{theorem}

\begin{proof}
Clearly, (UC1) is just from the definition of $\mathcal{C}_{l}$.

Then, for any $V_{1}, V_{2}\in \mathscr U$, there exists $V\in \mathscr U$ such that $V\subset V_{1}\cap V_{2}$, so $\mathcal{C}_{l}$ has the property (UC2).

Next, we prove that $\mathcal{C}_{l}$ has the property (UC3), then it suffices to show the following claim.

\smallskip
{\bf Claim:} For every $V\in \mathscr U$ there exists $V_{1}\in \mathscr U$ such that $\mbox{st}(x\oplus V_{1}, \mathcal{C}_{l}(V_{1}))\subset x\oplus V$, for each $x\in G.$

\smallskip
Indeed, since the formula $f(x, y, z)=(x\oplus (\ominus y))\oplus z$ defines a continuous mapping $f: G\times G\times G\rightarrow G$ and $f(0, 0, 0)=0$, for every $V\in \mathscr U$ there exists a $V_{1}\in \mathscr U$ such that $f(V_{1}\times V_{1}\times V_{1})=(V_{1}\oplus V_{1})\oplus V_{1}\subset V$. Take an arbitrary $x\in G$, and if $(x\oplus V_{1})\cap (x_{1}\oplus V_{1})\not =\emptyset$, then there exist $v_{1}, v_{2}\in V_{1}$ such that $x\oplus v_{1}=x_{1}\oplus v_{2}$. Therefore,
\begin{eqnarray}
x_{1}&=&(x_{1}\oplus v_{2})\oplus gyr[x_{1}, v_{2}](\ominus v_{2})\nonumber\\
&=&(x\oplus v_{1})\oplus gyr[x_{1}, v_{2}](\ominus v_{2})\nonumber\\
&\in &(x\oplus v_{1})\oplus gyr[x_{1}, v_{2}](V_{1})\nonumber\\
&\subset &(x\oplus V_{1})\oplus V_{1}\nonumber\\
&\subset&x\oplus (V_{1}\oplus gyr[V_{1}, x](V_{1}))\nonumber\\
&=&x\oplus (V_{1}\oplus V_{1}).\nonumber
\end{eqnarray}
Then for any $v\in V_{1}$ we have
\begin{eqnarray}
x_{1}\oplus v&\in &(x\oplus (V_{1}\oplus V_{1}))\oplus v\nonumber\\
&\subset&x\oplus ((V_{1}\oplus V_{1})\oplus gyr[V_{1}\oplus V_{1},x](v))\nonumber\\
&\subset &x\oplus ((V_{1}\oplus V_{1})\oplus gyr[V_{1}\oplus V_{1},x](V_{1}))\nonumber\\
&=&x\oplus ((V_{1}\oplus V_{1})\oplus V_{1})\nonumber\\
&\subset &x\oplus V,\nonumber
\end{eqnarray}
which implies that the claim holds.

Finally, take any $x, y\in G$ such that $x\not =y$. Clearly, there exists $V\in \mathscr U$ such that $\ominus x\oplus y\not \in V$. We show that no member of the cover $\mathcal{C}_{l}(V_{1})$, where $V_{1}\in \mathscr U$ satisfies $V_{1}\oplus V_{1}\subset V$, contains both $x$ and $y$. Suppose not, then there exists $x_{0}\in G$ such that $x\in x_{0}\oplus V_{1}$ and $y\in x_{0}\oplus V_{1}$. So, there exists $v_{1}\in V_{1}$ such that $x_{0}\oplus v_{1}=x$. Then $$x_{0}=x\oplus gyr[x_{0}, v_{1}](\ominus v_{1})\in x\oplus gyr[x_{0}, v_{1}](V_{1})=x\oplus V_{1}.$$ It follows that $$y\in (x\oplus V_{1})\oplus V_{1}=x\oplus (V_{1}\oplus gyr[V_{1}, x](V_{1}))=x\oplus (V_{1}\oplus V_{1}).$$ Hence, $(\ominus x)\oplus y\in V_{1}\oplus V_{1}\subset V$, which contradicts with the choice of $V$. Therefore, $\mathcal{C}_{l}$ has property (UC4).

Since the collection $\mathcal{C}_{l}$ consists of open covers of $G$, it just need to observe that for every $x\in G$ and neighborhood $U$ of $x$ there exists an $V\in \mathscr U$ such that $x\oplus V\subset U$. Therefore, the topology induced by the uniformity coincides with the original topology of $G$.
\end{proof}

\begin{lemma}\cite{BL1}\label{lllll}
Every strongly topological gyrogroup is a Tychonoff space.
\end{lemma}

\begin{lemma}\label{yl31}
Suppose that $(G,\tau ,\oplus)$ is a strongly topological gyrogroup with a symmetric open neighborhood base $\mathscr U$ at $0$. Then, for each $U\in \mathscr U$, there exists a locally finite open cover $\mathscr{U}$ which refines the cover $\mathcal{C}_{l}(U))$.
\end{lemma}

\begin{proof}
By the claim of the proof of Theorem~\ref{th9}, there exists $V\in \mathscr U$ such that $\mathcal{C}_{l}(V)$ is a star refinement of $\mathcal{C}_{l}(U)$. From Lemma~\ref{lllll}, \cite[Lemma 2]{M1953} and \cite[Theorem 1]{ST1948}, we can see that there exists a locally finite open cover $\mathscr{U}$ which refines the cover $\mathcal{C}_{l}(U))$.
\end{proof}

\begin{lemma}\label{yl32}
Suppose that $(G, \tau, \oplus)$ is a strongly topological gyrogroup with a symmetric open neighborhood base $\mathscr U$ at $0$, $U$ is a non-empty open subset of $G$. Then there exists a locally finite open covering $\gamma$ of $G$ such that $\overline{V}$ is homeomorphic to a closed subspace of $\overline{U}$, for every $V\in \gamma$.
\end{lemma}

\begin{proof}
Let $\xi=\{x\oplus U: x\in G\}$. Since $\mathscr U$ is an open neighborhood base at $0$, there exists $V\in \mathscr U$ such that $\eta=\{x\oplus V\}_{x\in G}$ refines $\xi$. We observe that $\eta$ is an open covering of $G$. By Lemma \ref{yl31}, there exists a locally finite open covering $\gamma$ refining $\eta$. Since the closure of $x\oplus U$ is homeomorphic to the closure of $U$, then $\gamma$ is the covering we need.
\end{proof}

The following lemma is well known, see \cite{M1953}.

\begin{lemma}\cite{M1953}\label{yl33}
Let $X$ be a topological space, and $\gamma$ be a locally finite open covering of $X$ such that $\overline{V}$ is paracompact for each $V\in \gamma$. Then $X$ is paracompact.
\end{lemma}

Since every closed subspace of a paracompact space is also paracompact, it follows from Lemma \ref{yl32} and Lemma \ref{yl33} that we have the following main result in this section.

\begin{theorem}\label{th10}
Every locally paracompact strongly topological gyrogroup is paracompact.
\end{theorem}

\begin{corollary}\cite{AU}
Every locally paracompact topological group is paracompact.
\end{corollary}

From Corollary \ref{tl1} and Theorem~\ref{th10}, we also have the following corollary.

\begin{corollary}
Let $(G, \tau, \oplus)$ be a strongly topological gyrogroup with a symmetric neighborhood base $\mathscr U$ at $0$, and let $H$ be a locally compact admissible subgyrogroup generated from $\mathscr U$. If the quotient space $G/H$ is locally paracompact, then $G$ is paracompact.
\end{corollary}

\smallskip
{\bf Acknowledgements}. We wish to thank anonymous referees and professor Jiling Cao for the detailed list of corrections, suggestions to the paper, and all her/his efforts
in order to improve the paper.

\end{document}